\RequirePackage{ifpdf}
\ifpdf 
\documentclass[pdftex]{sigma}
\else
\documentclass{sigma}
\fi

\usepackage{upgreek}

\usepackage{bm}
\usepackage{color}
\usepackage{amsthm}
\usepackage{mathrsfs}

\DeclareSymbolFont{AMSb}{U}{msb}{m}{n}
\DeclareSymbolFontAlphabet{\mathbb}{AMSb}

\newcommand{\E}{\mathscr{E}}
\newcommand\dist{\,{\rm dist}\,}

\newcommand\supp{\mathop{\rm supp}}

\newcommand\p{\partial}
\newcommand{\at}[1]{\vert\sb{\sb{#1}}}

\def\Re{{\rm Re\, }}
\def\Im{{\rm Im\,}}
\providecommand\C{{\mathbb C}}
\renewcommand\C{{\mathbb C}}
\newcommand{\R}{{\mathbb R}}
\newcommand{\N}{{\mathbb N}}

\newcommand{\Abs}[1]{\left\vert#1\right\vert}
\newcommand{\abs}[1]{\vert #1 \vert}
\newcommand{\Norm}[1]{\left\Vert #1 \right\Vert}
\newcommand{\norm}[1]{\Vert #1 \Vert}
\newcommand{\const}{{\rm const}}
\newcommand\cnst{\mathop{C}}
\newcommand\sothat{{\rm :}\ }

\providecommand{\ltor}[1]{
\ifnum #1=1{\it i}\else\ifnum #1=2{\it ii}\else\ifnum #1=3{\it iii}
\else\ifnum #1=4 {\it iv}\fi\fi\fi\fi
}

\DeclareMathSymbol{\varOmega}{\mathord}{letters}{"0A}

\theoremstyle{plain}

\theoremstyle{definition}

\theoremstyle{remark}

\numberwithin{equation}{section}

\begin{document}

\allowdisplaybreaks

\renewcommand{\PaperNumber}{010}

\FirstPageHeading

\renewcommand{\thefootnote}{$\star$}

\ShortArticleName{Global Attraction to Solitary Waves}

\ArticleName{Global Attraction to Solitary Waves\\ in Models Based on the Klein--Gordon Equation\footnote{This
paper is a contribution to the Proceedings of the Seventh
International Conference ``Symmetry in Nonlinear Mathematical
Physics'' (June 24--30, 2007, Kyiv, Ukraine). The full collection
is available at
\href{http://www.emis.de/journals/SIGMA/symmetry2007.html}{http://www.emis.de/journals/SIGMA/symmetry2007.html}}}

\Author{Alexander I. KOMECH~$^{\dag\S}$ and Andrew A. KOMECH~$^{\ddag\S}$}

\AuthorNameForHeading{A.I.~Komech and A.A.~Komech}
\Address{$^\dag$~Faculty of Mathematics, University of Vienna, Wien A-1090, Austria}
\EmailD{\href{mailto:alexander.komech@univie.ac.at}{alexander.komech@univie.ac.at}}
\URLaddressD{\url{http://www.mat.univie.ac.at/~komech/}}

\Address{$^\ddag$~Mathematics Department, Texas A\&M University, College Station, TX 77843, USA}
\EmailD{\href{mailto:comech@math.tamu.edu}{comech@math.tamu.edu}}
\URLaddressD{\url{http://www.math.tamu.edu/~comech/}}

\Address{$^\S$~Institute for Information Transmission Problems,
B. Karetny 19, Moscow 101447, Russia}

\ArticleDates{Received November 01, 2007, in f\/inal form January
22, 2008; Published online January 31, 2008}

\Abstract{We review recent results
on global attractors of $\mathbf{U}(1)$-invariant
dispersive Hamiltonian systems.
We study several models
based on the Klein--Gordon equation
and sketch the proof that in these models,
under certain generic assumptions,
the weak global attractor is represented
by the set of all solitary waves.
In general,
the attractors may also contain multifrequency solitary waves;
we give examples of systems
which contain such solutions.}


\Classification{35B41; 37K40; 37L30; 37N20; 81Q05}

\section{Introduction}
\label{sect-introduction}

The long time asymptotics
for nonlinear wave equations
have been the subject of intensive research,
starting with the pioneering papers by
Segal
\cite{MR0153967,MR0152908},
Strauss \cite{MR0233062},
and Morawetz and Strauss \cite{MR0303097},
where
the nonlinear scattering and local attraction to zero were considered.
Global attraction (for large initial data) to zero
may not hold if there are
\emph{quasistationary solitary wave solutions}
of the form
\begin{gather}\label{sw00}
\psi(x,t)=\phi(x)e^{-i\omega t},
\qquad
{with}
\quad
\omega\in\R,
\quad \lim\sb{\abs{x}\to\infty}\phi(x)=0.
\end{gather}
We will call such solutions \emph{solitary waves}.
Other appropriate names are \emph{nonlinear eigenfunctions}
and \emph{quantum stationary states}
(the solution (\ref{sw00}) is not exactly stationary,
but certain observable quantities, such as the charge
and current densities, are time-independent indeed).

Existence of such solitary waves
was addressed by Strauss in \cite{MR0454365},
and then the orbital stability of solitary waves
has been studied
in \cite{MR901236,MR723756,MR792821,MR804458}.
The asymptotic stability of solitary waves
has been studied by
Sof\/fer and Weinstein \cite{MR1071238,MR1170476},
Buslaev and Perelman \cite{MR1199635e,MR1334139},
and then by others.

The existing results suggest that the set of
orbitally stable solitary waves
typically forms a~\emph{local attractor},
that is to say, attracts any
f\/inite energy solutions
that were initially close to it.
Moreover, a natural hypothesis is that
the set for all solitary waves forms a
\emph{global attractor}
of all f\/inite energy solutions.
We address this question in the present paper,
reviewing the results on
the global attraction in several models
based on the Klein--Gordon equation,
and describing the developed techniques.


We brief\/ly discuss
the long-time solitary wave asymptotics
for $\mathbf{U}(1)$-invariant Hamiltonian systems
in Section~\ref{sect-history}.
The def\/initions and results
on global attraction to solitary waves
from the recent papers \cite{ubk-arma,ukk-mpi,ukr-mpi}
are presented
in Section~\ref{sect-results}.
We also give there a very brief sketch of the proof.
In Section~\ref{sect-proof},
we give a description of all the steps
(omitting excessive technical points)
of the argument for the simplest model:
Klein--Gordon equation interacting with a nonlinear oscillator.
The key parts of the proof are presented in full detail.
The examples of (untypical) multifrequency solitary waves
are given in Section~\ref{sect-examples}.

\section[History of solitary asymptotics
for $\mathbf{U}(1)$-invariant Hamiltonian systems]{History of solitary asymptotics\\
for $\boldsymbol{\mathbf{U}(1)}$-invariant Hamiltonian systems}
\label{sect-history}

\subsection{Quantum theory}

\subsubsection*{Bohr's stationary orbits
as solitary waves}

Let us focus on the behavior
of the electron in the Hydrogen atom.
According to Bohr's postulates \cite{bohr1913},
an unperturbed electron runs forever
along
certain \emph{stationary orbit},
which we denote~$\vert E\rangle$
and call \emph{quantum stationary state}.
Once in such a state,
the electron has a f\/ixed value of ener\-gy~$E$,
not losing the energy via emitting radiation.
The electron can jump from one
quantum stationary state to another,
\begin{gather}\label{transitions}
\vert E\sb{-}\rangle
\longmapsto
\vert E\sb{+}\rangle,
\end{gather}
emitting or absorbing
a quantum of light with the energy
equal to the dif\/ference of
the energies~$E\sb{+}$ and~$E\sb{-}$.
The old quantum theory
was based on the quantization condition
\begin{gather}\label{oqt-qc}
\oint\mathbf{p}\cdot d\mathbf{q}=2\pi\hbar n,
\qquad n\in\N.
\end{gather}
This condition leads to the values
\begin{gather*}
E\sb n=-\frac{m{\rm e}^4}{2\hbar^2 n^2},
\qquad
n\in\N,
\end{gather*}
for the energy levels in Hydrogen,
in a good agreement with the experiment.
Apparently, the condition (\ref{oqt-qc})
did not explain the perpetual
circular motion of the electron.
According to the classical Electrodynamics,
such a motion would be accompanied
by the loss of energy via radiation.

In terms of the wavelength
$\lambda=\frac{2\pi\hbar}{\abs{\mathbf{p}}}$
of de Broglie's \emph{phase waves}
\cite{deBroglie1924},
the condition (\ref{oqt-qc})
states that
the length of the classical orbit of the electron
is the integer multiple of $\lambda$.
Following de Broglie's ideas,
Schr\"odinger identif\/ied
Bohr's \emph{stationary orbits},
or quantum stationary states~$\vert E\rangle$,
with the wave functions
that have the form
\begin{gather*}
\psi(\mathbf{x},t)
=\phi\sb\omega(\mathbf{x})e\sp{-i\omega t},
\qquad
\omega=E/\hbar,
\end{gather*}
where $\hbar$ is Planck's constant.
Physically,
the charge and current densities
\begin{gather*}
\rho(\mathbf{x},t)={\rm e}\bar\psi\psi,
\qquad
\mathbf{j}(\mathbf{x},t)=\frac{{\rm e}}{2i}
(\bar\psi\cdot\nabla\psi-\nabla\bar\psi\cdot\psi),
\end{gather*}
with ${\rm e}<0$ being the charge of the electron,
which correspond to
the (quasi)stationary states of the form
$\psi(\mathbf{x},t)
=\phi\sb\omega(\mathbf{x})e^{-i\omega t}$
do not depend on time,
and therefore the generated electromagnetic f\/ield
is also stationary
and does not carry the energy away from the system,
allowing the electron cloud to f\/low forever
around the nucleus.

\subsubsection*{Bohr's transitions as global attraction to solitary waves}

Bohr's second postulate
states that the electrons can jump from one
quantum stationary state
(Bohr's \emph{stationary orbit}) to another.
This postulate
suggests the dynamical interpretation
of Bohr's transitions
as long-time attraction
\begin{gather}\label{ga}
\Psi(t)\longrightarrow\vert E\sb\pm\rangle,
\qquad
t\to\pm\infty
\end{gather}
for any trajectory $\Psi(t)$ of the corresponding dynamical system,
where the limiting states $\vert E\sb\pm\rangle$ generally
depend on the trajectory.
Then the quantum stationary states
$\mathcal{S}\sb{0}$
should be viewed as the points of the \emph{global attractor}
$\mathscr{A}$.



The attraction (\ref{ga})
takes the form of the long-time asymptotics
\begin{gather}\label{asymptotics}
\psi(x,t)
\sim
\phi\sb{\omega\sb\pm}(x)e\sp{-i\omega\sb{\pm}t},
\qquad
t\to\pm\infty,
\end{gather}
that hold for each f\/inite energy solution.
However, because of the superposition
principle,
the asymptotics of type (\ref{asymptotics})
are generally impossible for
the linear autonomous
equation,
be it the
Schr\"o\-din\-ger equation
\begin{gather*}
i\hbar\p\sb t
\psi
=-\frac{\hbar^2}{2m}\Delta\psi
-\frac{{\rm e}^2}{\abs{\mathbf{x}}}\psi,
\end{gather*}
or relativistic Schr\"odinger
or Dirac equation in the Coulomb f\/ield.
An adequate description of this process
requires to consider the equation
for the electron wave function
(Schr\"o\-din\-ger or Dirac equation)
coupled to the Maxwell system
which governs the time evolution of the four-potential
$A(x,t)=(\varphi(x,t),\mathbf{A}(x,t))$:
\begin{gather*}
(i\hbar\p\sb t-{\rm e}\varphi)^2\psi
=\left(c\frac{\hbar}{i}\nabla-{\rm e}\mathbf{A}\right)^2\psi+m^2 c^4\psi,\nonumber
\\
\square\varphi
=4\pi{\rm e}
(\bar\psi\psi-\delta(\mathbf{x})),
\qquad
\square\mathbf{A}
=4\pi{\rm e}
\frac
{\bar\psi\cdot\nabla\psi
-\nabla\bar\psi\cdot\psi}{2i}.
\end{gather*}
Consideration of such a system
seems inevitable,
because, again by Bohr's postulates,
the transitions (\ref{transitions})
are followed by electromagnetic radiation
responsible for the atomic spectra.
Moreover,
the Lamb shift
(the energy of $2S\sb{1/2}$ state being slightly higher
than the energy of $2P\sb{1/2}$ state)
can not be explained in terms
of the linear Dirac equation in the external Coulomb f\/ield.
Its theoretical explanation
within the Quantum Electrodynamics
takes into account the higher order
interactions of
the electron wave function
with the electromagnetic f\/ield,
referred to as the vacuum polarization and the electron self-energy correction.

The coupled Maxwell--Schr\"o\-din\-ger system
was initially introduced in \cite{Sch81109}.
It is a $\mathbf{U}(1)$-invariant
nonlinear Hamiltonian system.
Its global well-posedness was considered in \cite{MR1331696}.
One might expect the following
generalization of asymptotics (\ref{asymptotics})
for solutions to the coupled Maxwell--Schr\"o\-din\-ger
(or Maxwell--Dirac)
equations:
\begin{gather}\label{asymptotics-a}
(\psi(x,t), A(x,t))
\sim
\left(
\phi\sb{\omega\sb\pm}(x)e\sp{-i\omega\sb\pm t},
A\sb{\omega\sb\pm}(x)
\right),
\qquad
t\to\pm\infty.
\end{gather}
The asymptotics (\ref{asymptotics-a}) would mean that
the set of all solitary waves
\[
\big\{\left(\phi\sb\omega e^{-i\omega t},
A\sb\omega\right):\omega\in\R\big\}
\]
forms a global attractor for the coupled system.
The asymptotics of this form are not available yet
in the context of coupled systems.
Let us mention that
the existence of the solitary  waves for
the coupled Maxwell--Dirac equations
was established in \cite{MR1386737}.

\subsection{Solitary waves as global attractors for dispersive systems}

Convergence to a global attractor
is well known for dissipative systems, like Navier--Stokes equations
(see \cite{MR1156492,He81,MR1441312}).
For such systems,
the global attractor is formed by the \emph{static stationary states},
and the corresponding asymptotics (\ref{asymptotics})
only hold for $t\to+\infty$.

We would like to know whether dispersive Hamiltonian systems
could, in the same spirit,
possess f\/inite dimensional global attractors,
and whether such attractors
are formed by the solitary waves.
Although there is no dissipation per se,
we expect that
the attraction is caused by
certain friction mechanism
via the dispersion (local energy decay).
Because of the dif\/f\/iculties posed by the system of interacting
Maxwell and Dirac (or Schr\"odinger) f\/ields
(and, in particular, absence of the a priori
estimates for such systems),
we will work with simpler models that share
certain key properties of the coupled Maxwell--Dirac
or Maxwell--Schr\"odinger systems.
Let us try to single out these key features:
\begin{enumerate}\itemsep=0pt
\item
\emph{The system is $\mathbf{U}(1)$-invariant.}\\
This invariance leads
to the existence of solitary wave solutions
$\phi\sb\omega(x)e^{-i\omega t}$.

\item
\emph{The linear part of the system has a dispersive character.}\\
This property provides certain dissipative features
in a Hamiltonian system,
due to local energy decay via the dispersion mechanism.
\item
\emph{The system is nonlinear.}\\
The nonlinearity is needed for the convergence to a single state
of the form $\phi\sb\omega(x)e^{-i\omega t}$.
Bohr type transitions
to pure eigenstates of the energy operator
are impossible in a linear system
because of the superposition principle.
\end{enumerate}
We suggest that these are the very features
are responsible for the global attraction,
such as~(\ref{asymptotics}), (\ref{asymptotics-a}),
to ``quantum stationary states''.

\begin{remark}
The global attraction
(\ref{asymptotics}), (\ref{asymptotics-a})
for $\mathbf{U}(1)$-invariant equations
suggests the corresponding extension
to general $\mathbf{G}$-invariant equations
($\mathbf{G}$ being the Lie group):
\begin{gather}\label{asymptotics-g}
\psi(x,t)
\sim
\psi\sb\pm(x,t)=e\sp{\bm\Omega\sb{\pm}t}\phi\sb\pm(x),
\qquad
t\to\pm\infty,
\end{gather}
where $\bm\Omega\sb{\pm}$ belong to the corresponding Lie algebra
and $e\sp{\bm\Omega\sb{\pm}t}$
are the one-parameter subgroups.
Respectively,
the global attractor would consist of the solitary waves
(\ref{asymptotics-g}).
On a seemingly related note,
let us mention that according to Gell-Mann--Ne'eman theory \cite{GN64}
there is a cor\-respondence
between the Lie algebras
and the classif\/ication of the elementary particles
which are the ``quantum stationary states''.
The correspondence
has been conf\/irmed
experimentally by the discovery of the omega-minus Hyperon.
\end{remark}

Besides Maxwell--Dirac system,
naturally, there are various nonlinear systems under
consideration in the Quantum Physics.
One of the simpler nonlinear models is
the nonlinear Klein--Gordon equation
which takes its origin from the articles by Schif\/f
\cite{PhysRev.84.1,PhysRev.84.10},
in his research
on the classical nonlinear meson theory of nuclear forces.
The mathematical analysis of this equation
is started by J\"orgens \cite{MR0130462}
and Segal \cite{MR0153967},
who studied its global well-posedness in the energy space.
Since then, this equation
(alongside with the
nonlinear Schr\"odinger equation)
has been the main playground
for developing tools to handle
more general nonlinear Hamiltonian systems.
The nonlinear Klein--Gordon equation
is a natural candidate
for exhibiting solitary asymptotics~(\ref{asymptotics}).

Now let us describe the existing results on attractors
in the context of dispersive Hamiltonian systems.

\subsubsection*{Local and global attraction to zero}

The asymptotics of type (\ref{asymptotics})
were discovered f\/irst with $\psi\sb\pm=0$
in the scattering theory.
Namely, Segal, Morawetz, and Strauss
studied the (nonlinear) scattering
for solutions of nonlinear Klein--Gordon equation in $\R^3$
\cite{MR0303097,MR0217453,MR0233062}.
We may interpret these results
as \emph{local}
(referring to small initial data)
attraction to zero:
\begin{gather}\label{attraction-0}
\psi(x,t)\sim\psi\sb\pm=0,\qquad
t\to\pm\infty.
\end{gather}
The asymptotics (\ref{attraction-0})
hold on an arbitrary compact set
and represent the well-known local energy decay.
These results were further extended in
\cite{MR824083,MR535231,MR1120284,MR654553}.
Apparently, there could be no \emph{global}
attraction to zero
(\emph{global} referring to arbitrary initial data)
if there are solitary wave solutions
$\phi\sb\omega(x)e\sp{-i\omega t}$.

\subsubsection*{Solitary waves}

The existence of solitary wave solutions
of the form
\begin{gather*}
\psi\sb\omega(x,t)=\phi\sb\omega(x)e\sp{-i\omega t},
\qquad
\omega\in\R,
\quad\phi\sb\omega\in H\sp 1(\R^n),
\end{gather*}
with $H\sp{1}(\R^n)$ being the Sobolev space,
to the nonlinear Klein--Gordon equation (and nonlinear
Schr\"odinger equation)
in $\R^n$,
in a rather generic situation,
was established in \cite{MR0454365}
(a more general result
was obtained in \cite{MR695535,MR695536}).
Typically, such solutions exist for
$\omega$ from an interval or a collection of intervals
of the real line.
We denote the set of all solitary waves by ${\mathcal{S}\sb{0}}$.

Due to the $\mathbf{U}(1)$-invariance
of the equations,
the factor-space $\mathcal{S}\sb{0}/\mathbf{U}(1)$
in a generic situation
is isomorphic to a f\/inite union of intervals.
Let us mention
that there are numerous results on the existence of
solitary wave solutions
to nonlinear Hamiltonian systems with $\mathbf{U}(1)$ symmetry.
See e.g. \cite{MR765961,MR847126,MR1344729}.


While all localized stationary solutions
to the nonlinear wave equations
in spatial dimensions $n\ge 3$ turn out to be unstable
(the result known as ``Derrick's theorem'' \cite{MR30:4510}),
\emph{quasistationary} solitary waves can be orbitally stable.
Stability of solitary waves
takes its origin from \cite{VaKo}
and has been extensively studied by Strauss and his school
in \cite{MR901236,MR723756,MR792821,MR804458}.

\subsubsection*{Local attraction to solitary waves}

First results on the asymptotics of type (\ref{asymptotics})
with $\omega\sb\pm\ne 0$
were obtained for nonlinear $\mathbf{U}(1)$-invariant
Schr\"o\-din\-ger equations
in the context of asymptotic stability.
This establishes asymptotics of type~(\ref{asymptotics})
but only for solutions close to the solitary waves,
proving the existence of a \emph{local attractor}.
This was f\/irst done
by Sof\/fer and Weinstein and by Buslaev and Perelman
in~\cite{MR1199635e,MR1334139,MR1071238,MR1170476},
and then developed in~\cite{MR1972870,MR1893394,MR1835384,MR2027616,MR1488355,MR1681113}
and other papers.

\subsubsection*{Global attraction to solitary waves}

The \emph{global attraction}
of type (\ref{asymptotics})
with $\psi\sb\pm\ne 0$ and $\omega\sb{\pm}=0$
was established in certain models
in
\cite{MR1726676,MR1748357,MR1434147,MR1412428,MR1203302e,MR1359949}
for a number of nonlinear
wave problems.
There the attractor
is the set of all \emph{static} stationary states.
Let us mention that this set could be inf\/inite
and contain continuous components.

In \cite{MR2032730} and \cite{ubk-arma},
the attraction to the set of solitary waves
(see Fig.~\ref{fig-attractor})
is proved
for the Klein--Gordon f\/ield coupled to a nonlinear oscillator.
In \cite{ukk-mpi}, this result has been generalized
for the Klein--Gordon f\/ield coupled to several oscillators.
The paper~\cite{ukr-mpi} gives the extension to the higher-dimensional setting
for a model with the nonlinear self-interaction
of the mean f\/ield type.
We are going to describe these results in this survey.

\begin{figure}[htbp]
\centerline{\includegraphics{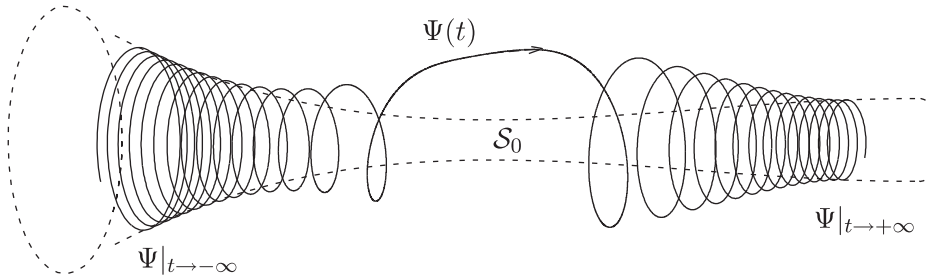}}

\caption{For $t\to\pm\infty$,
a f\/inite energy solution $\Psi(t)$
approaches the global attractor
$\mathscr{A}$
which coincides with the set of all solitary waves
$\mathcal{S}\sb 0$.}
\label{fig-attractor}
\end{figure}

We are aware of but one recent advance \cite{MR2304091}
in the f\/ield
of nontrivial (nonzero) global attractors for Hamiltonian PDEs.
In that paper, the global attraction for the nonlinear Schr\"odinger equation
in dimensions $n\ge 5$ was considered.
The dispersive (outgoing) wave was explicitly specif\/ied
using the rapid decay of local energy in higher dimensions.
The global attractor was proved to be compact, but it was
neither identif\/ied with the set of solitary waves nor was proved
to be f\/inite-dimensional \cite[Remark 1.18]{MR2304091}.

\section{Assumptions and results}
\label{sect-results}

In \cite{ubk-arma,ukk-mpi,ukr-mpi}
we introduce the models
which possess the key properties we mentioned above:
$\mathbf{U}(1)$-invariance, dispersive character,
and the nonlinearity.
The models allow to prove the global attraction to solitary
waves and to develop certain techniques
which we hope will allow us to approach more general systems.

\subsection*{Model 1: Klein--Gordon f\/ield with a nonlinear oscillator}

We consider the Cauchy problem for the Klein--Gordon equation
with the nonlinearity concentrated at the origin:
\begin{gather}
\ddot\psi(x,t)
=\psi''(x,t)-m^2\psi(x,t)
+\delta(x)F(\psi(0,t)),
\qquad
x\in\R,
\nonumber\\
\psi\at{t=0}=\psi\sb{0}(x),
\qquad
\dot\psi\at{t=0}=\pi\sb{0}(x).\label{kg-1o}
\end{gather}
Above,
$m>0$
and
$F$ is a function
describing an oscillator at the point $x=0$.
The dots stand for the derivatives in $t$,
and the primes for the derivatives in $x$.
All derivatives and the equation are understood in
the sense of distributions.
We assume that equation (\ref{kg-1o}) is
$\mathbf{U}(1)$-invariant;
that is,
\begin{gather}\label{inv-f}
F(e\sp{i\theta}\psi)=e\sp{i\theta}F(\psi),
\qquad\theta\in\R,\quad\psi\in\C.
\end{gather}
If we identify a complex number $\psi=u+i v\in\C$
with the two-dimensional vector
$(u,v)\in\R\sp 2$,
then, physically, equation (\ref{kg-1o}) describes small crosswise
oscillations of the inf\/inite
string in three-dimensional space
$(x,u,v)$
stretched along the $x$-axis.
The string is subject to
the action of
an ``elastic force'' $-m^2\psi(x,t)$ and
coupled to a nonlinear oscillator
of force $F(\psi)$
attached at the point $x=0$.
We assume that
the oscillator force $F$ admits a real-valued potential,{\samepage
\begin{gather*}
F(\psi)=-\nabla\sb\psi U(\psi),\qquad\psi\in\C,
\quad
U\in C\sp 2(\C),
\end{gather*}
where the gradient is taken
with respect to $\Re\psi$ and $\Im\psi$.}

\begin{remark}
Viewing the model as an inf\/inite string in $\R\sp{3}$,
the assumption (\ref{inv-f})
means that the potential $U(\psi)$
is rotation-invariant with respect to the $x$-axis.
\end{remark}
The codes for the numerical simulation
of a f\/inite portion of a string
coupled to an oscillator,
with transparent boundary conditions
($\p\sb t\psi=\pm\p\sb x\psi$),
are available in \cite{kg-string}.

\subsection*{Model 2: Klein--Gordon f\/ield with
several nonlinear oscillators}

More generally,
we consider the Cauchy problem for the Klein--Gordon equation
with the nonlinearity concentrated at the points
$X\sb{1}<X\sb{2}<\dots<X\sb{N}$:
\begin{gather}
\ddot\psi(x,t)
=\psi''(x,t)-m^2\psi(x,t)
+\sum\sb{J=1}\sp{N}
\delta(x-X\sb{J})F\sb{J}(\psi(X\sb{J},t)),
\qquad
x\in\R,
\nonumber\\
\psi\at{t=0}=\psi\sb{0}(x),
\qquad
\dot\psi\at{t=0}=\pi\sb{0}(x).\label{kg-no}
\end{gather}

\subsection*{Model 3: Klein--Gordon f\/ield
with the mean f\/ield interaction}

We also consider the Klein--Gordon equation
with the mean f\/ield interaction:
\begin{gather}
\ddot\psi(x,t)
=\Delta\psi(x,t)-m^2\psi(x,t)
+\rho(x)F(\langle\rho,\psi(\cdot,t)\rangle),
\qquad
x\in\R^n,
\quad n\ge 1,
\nonumber\\
\psi\at{t=0}=\psi\sb 0(x),
\qquad
\dot\psi\at{t=0}=\pi\sb 0(x).\label{kg-mf}
\end{gather}
Above, $\rho$ is a smooth
real-valued
coupling function
from the Schwartz class:
$\rho\in\mathscr{S}(\R^n)$, $\rho\not\equiv 0$,
and
\[
\langle\rho,\psi(\cdot,t)\rangle=\int\sb{\R^n}
\rho(x)\psi(x,t)\,dx.
\]

\subsection*{Hamiltonian structure}

Equations (\ref{kg-no}),
(\ref{kg-mf})
formally can be written as a Hamiltonian system,
\begin{gather}\label{kg-h}
\dot\Psi(t)=\mathcal{J}\,D\mathcal{H}(\Psi),
\qquad
\mathcal{J}=\left[\begin{array}{cc}0&1\\-1&0\end{array}\right],
\end{gather}
where
$\Psi=(\psi,\pi)$
and
$D\mathcal{H}$ is the Fr\'echet derivative of the Hamilton functionals
\begin{gather*}
\mathcal{H}\sb{\rm osc}(\psi,\pi)
=\frac 1 2
\int\limits\sb{\R}
\left(
\abs{\pi}\sp 2+\abs{\psi'}\sp 2+m^2\abs{\psi}\sp 2
\right)
dx
+\sum\sb{J} U\sb{J}(\psi(X\sb{J})),
\\
\mathcal{H}\sb{\rm m.f.}(\psi,\pi)
=\frac 1 2
\int\limits\sb{\R^n}
\left(
\abs{\pi}\sp 2+\abs{\nabla\psi}\sp 2+m^2\abs{\psi}\sp 2
\right)
dx
+U(\langle\rho,\psi\rangle).
\end{gather*}
Since (\ref{kg-no}) and (\ref{kg-mf})
are $\mathbf{U}(1)$-invariant,
the N\"other theorem
formally implies that the value of the {\it charge functional}
\begin{gather*}
\mathcal{Q}(\psi,\pi)
=\frac{i}{2}\int
\left(\overline\psi\pi
-\overline\pi\psi\right) dx
\end{gather*}
is conserved for solutions
$\Psi(t)=(\psi(t),\pi(t))$ to (\ref{kg-h}).

Let us introduce
the phase space ${\E}$
of f\/inite energy states for equations
(\ref{kg-no}), (\ref{kg-mf}).
Denote by
$\norm{\cdot}\sb{L\sp 2}$
the norm in the complex Hilbert space $L\sp 2(\R^n)$
and by
$\norm{\cdot}\sb{L\sp 2\sb R}$ the norm in
$L\sp 2(\mathbb{B}^n\sb R)$,
where $\mathbb{B}^n\sb R$ is a ball of radius $R>0$.
Denote by
$\norm{\cdot}\sb{H\sp{s}\sb{R}}$ the norm in the
Sobolev space $H\sp{s}(\mathbb{B}^n\sb{R})$
(which is the dual to the Sobolev space $H\sp{-s}\sb{0}(\mathbb{B}^n\sb{R})$
of functions supported in the ball of radius~$R$).

{\samepage

\begin{definition}[The phase space]
\label{definition-e}\

\begin{enumerate}\itemsep=0pt
\item
$\E=H\sp 1(\R^n)\oplus L\sp 2(\R^n)$,
$n\ge 1$,
is the Hilbert space of the states
$(\psi,\pi)$,
with the norm
\begin{gather*}
\norm{(\psi,\pi)}\sb{\E}^2
:=
\norm{ \pi}\sb{L\sp 2}^2
+\norm{\nabla\psi}\sb{L\sp 2}^2+m^2\norm{\psi}\sb{L\sp 2}^2.
\end{gather*}
\item
For $\epsilon\ge 0$, $\E\sp{-\epsilon}=H\sp{1-\epsilon}(\R^n)\oplus H\sp{-\epsilon}(\R^n)$
is the space with the norm
\begin{gather*}
\norm{(\psi,\pi)}\sb{\E\sp{-\epsilon}}
=
\norm{(1-\Delta)^{-\epsilon/2}(\psi,\pi)}\sb{\E}.
\end{gather*}
\item Def\/ine the seminorms
  \[
  \norm{(\psi,\pi)}\sb{\E\sp{-\epsilon},R}^2
  :=
  \norm{\pi}\sb{H\sp{-\epsilon}\sb R}^2
  +
  \norm{\nabla\psi}\sb{H\sp{-\epsilon}\sb R}^2
  +m^2\norm{\psi}\sb{H\sp{-\epsilon}\sb R}^2,
  \qquad R>0.
  \]
  $\E\sp{-\epsilon}\sb F$ is the space with the norm
  \begin{equation}\label{def-e-metric}
  \norm{(\psi,\pi)}\sb{\E\sb F\sp{-\epsilon}}
  =\sum\sb{R=1}\sp{\infty} 2^{-R}\norm{(\psi,\pi)}\sb{{\E\sp{-\epsilon},R}}.
  \end{equation}

\end{enumerate}
\end{definition}

}



Equations (\ref{kg-no}), (\ref{kg-mf})
are formally Hamiltonian systems
with the Hamilton functionals
$\mathcal{H}\sb{\rm osc}$
and
$\mathcal{H}\sb{\rm m.f.}$, respectively,
and with
the phase space ${\E}$
from Def\/inition~\ref{definition-e}
(for equation (\ref{kg-no}), the dimension is $n=1$).
Both
$\mathcal{H}\sb{\rm osc}$
(or $\mathcal{H}\sb{\rm m.f.}$)
and $\mathcal{Q}$ are continuous functionals on ${\E}$.

\subsection*{Global well-posedness}

\begin{theorem}[Global well-posedness]
\label{theorem-well-posedness}
Assume that the nonlinearity in \eqref{kg-no}
is given by
$F\sb J(z)=-\nabla U\sb J(z)$
with
$\inf\limits\sb{z\in\C}U\sb J(z)>-\infty$,
$1\le J\le N$
(or
$F(z)=-\nabla U(z)$
with
$\inf\limits\sb{z\in\C}U(z)>-\infty$
in \eqref{kg-mf}, respectively).
Then:
\begin{enumerate}\itemsep=0pt
\item
For every $(\psi\sb 0,\pi\sb 0)\in {\E}$
the Cauchy problem
\eqref{kg-no} (\eqref{kg-mf}, respectively)
has a unique global solution
$\psi(t)$ such that $(\psi,\dot\psi)\in C(\R,{\E})$.
\item
The map
$W(t):\;(\psi\sb 0,\pi\sb 0)\mapsto(\psi(t),\dot\psi(t))$
is continuous in ${\E}$
for each $t\in\R$.
\item
The energy and charge are conserved:
$\mathcal{H}(\psi(t),\dot\psi(t))=\const$,
$\mathcal{Q}(\psi(t),\dot\psi(t))=\const$,
$t\in\R$.
\item
The following \emph{a priori} bound holds:
\begin{gather}\label{a-priori}
\norm{(\psi(t),\dot\psi(t))}\sb{\E}
\le C(\psi\sb 0,\pi\sb 0),
\qquad
t\in\R.
\end{gather}
\item
For any $\epsilon\in[0,1]$,
\begin{gather*}
(\psi,\dot\psi)
\in C\sp{(\epsilon)}(\R,\E\sp{-\epsilon}),
\end{gather*}
where $C\sp{(\epsilon)}$ denotes the space of H\"older-continuous functions.
\end{enumerate}
\end{theorem}

The proof is contained in
\cite{ubk-arma, ukk-mpi}, and \cite{ukr-mpi}.

\subsection*{Solitary waves}

\begin{definition}[Solitary waves]
\label{def-solitary-waves} \

\begin{enumerate}\itemsep=0pt
\item
The solitary waves of equation (\ref{kg-no})
are solutions of the form
\begin{gather}\label{solitary-waves}
\psi(x,t)=\phi\sb\omega(x)e\sp{-i\omega t},
\qquad
{\rm where}
\quad
\omega\in\R,
\quad
\phi\sb\omega\in H\sp{1}(\R^n).
\end{gather}
\item
The set of all solitary waves
is
$
\mathcal{S}\sb{0}
=
\left\{
\phi\sb\omega
\sothat\omega\in\R,\ \phi\sb\omega\in H\sp 1(\R^n)
\right\}$.
\item
The solitary manifold
is the set
$
\mathcal{S}
=
\left\{
(\phi\sb\omega,-i\omega\phi\sb\omega)
\sothat\omega\in\R,\ \phi\sb\omega\in H\sp 1(\R^n)
\right\}
\subset\E$.
\end{enumerate}
\end{definition}

\begin{remark} \
\begin{enumerate}\itemsep=0pt
\item
The $\mathbf{U}(1)$ invariance of
(\ref{kg-no}) and (\ref{kg-mf})
implies that the sets
$\mathcal{S}\sb{0}$,
$\mathcal{S}$
are invariant under multiplication by $e\sp{i\theta}$,
$\theta\in\R$.
\item
Let us note that for any $\omega\in\R$
there is a zero solitary wave with
$\phi\sb\omega(x)\equiv 0$
since $F\sb{J}(0)=0$.
\end{enumerate}
\end{remark}

The following proposition provides a concise description
of all solitary wave solutions to (\ref{kg-1o}).

\begin{proposition}
\label{prop-solitons}
There are no nonzero solitary waves
for $\abs{\omega}\ge m$.

For a particular $\omega\in(-m,m)$,
there is a nonzero solitary wave solution
to \eqref{kg-1o}
if and only if
there exists $C\in\C\backslash \{0\}$
so that
\begin{gather}\label{kaka}
2\kappa(\omega)=F(C)/C.
\end{gather}
The solitary wave solution is given by
\begin{gather}\label{solitary-wave-profile}
\phi\sb\omega(x)
=C e^{-\kappa(\omega)\abs{x}},
\qquad
\kappa(\omega)=\sqrt{m^2-\omega^2}.
\end{gather}
\end{proposition}

\begin{remark}
There could be more than one value $C>0$
satisfying (\ref{kaka}).
\end{remark}

\begin{remark}\label{remark-zero}
By (\ref{solitary-wave-profile}),
$\omega=\pm m$ can not correspond to a nonzero solitary wave.
\end{remark}

\begin{proof}
When we substitute the ansatz $\phi\sb\omega e^{-i\omega t}$
into (\ref{kg-1o}),
we get the following relation:
\begin{gather}\label{NEP}
-\omega\sp 2\phi\sb\omega(x)
=
\phi\sb\omega''(x)-m^2\phi\sb\omega(x)
+\delta(x)F(\phi\sb\omega(x)),\qquad x\in\R.
\end{gather}
The phase factor $e^{-i\omega t}$
has been canceled out.
Equation (\ref{NEP}) implies that away from the origin
we have
\[
\phi\sb\omega''(x)=(m^2-\omega^2)\phi\sb\omega(x),
\qquad
x\ne 0,
\]
hence
$\phi\sb\omega(x)=C\sb\pm e\sp{-\kappa\sb\pm\abs{x}}$ for $\pm x>0$,
where $\kappa\sb\pm$
satisfy
$\kappa\sb\pm^2=m^2-\omega\sp 2$.
Since $\phi\sb\omega(x)\in H\sp 1$,
it is imperative that $\kappa\sb\pm>0$;
we conclude that
$\abs{\omega}<m$
and that
$\kappa\sb\pm=\sqrt{m^2-\omega^2}>0$.
Moreover, since
the function $\phi\sb\omega(x)$ is continuous,
$C\sb{-}=C\sb{+}=C\ne 0$
(since we are looking for nonzero solitary waves).
We see that
\begin{gather}\label{profile}
\phi\sb\omega(x)=C e\sp{-\kappa\abs{x}},
\qquad
C\ne 0,
\qquad
\kappa\equiv\sqrt{m^2-\omega^2}>0.
\end{gather}
Equation (\ref{NEP}) implies the following jump condition at $x=0$:
\begin{gather*}
0=\phi\sb\omega'(0+)-\phi\sb\omega'(0-)+F(\phi\sb\omega(0)),
\end{gather*}
which is satisf\/ied due to (\ref{kaka}) and (\ref{profile}).
\end{proof}

\subsection*{Global attraction to solitary waves}

We will combine the results
for the Klein--Gordon equation
with one and several oscillators ((\ref{kg-1o}) and (\ref{kg-no})).

\begin{theorem}[Global attraction
for (\ref{kg-no}), Klein--Gordon equation with $N$ oscillators]
\label{main-theorem-no}
Assume that
all the oscillators are strictly nonlinear:
for all $1\le J\le N$,
\begin{gather}\label{f-is-such-no}
F\sb{J}(\psi)=-\nabla U\sb{J}(\psi),
\\
{\rm where}\qquad
U\sb{J}(\psi)=\sum\limits\sb{l=0}\sp{p\sb{J}}u\sb{J,l}
\abs{\psi}\sp{2l},
\quad
u\sb{J,l}\in\R,
\qquad
u\sb{J,p\sb{J}}>0,
\quad
{\rm and}
\quad
p\sb{J}\ge 2.\nonumber
\end{gather}
Further,
if $N\ge 2$,
assume that
the intervals
$[X\sb{J},X\sb{J+1}]$, $\ 1\le J\le N-1$,
are small enough so that
\begin{gather}\label{delta-small}
\min\sb{1\le J\le N-1}
\left(\frac{\pi^2}{\abs{X\sb{J+1}-X\sb{J}}^2}+m^2\right)^{1/2}
>
m
\max\sb{1\le J\le N}
\min\left(
\prod\limits\sb{l=1}\sp{J}(2p\sb l-1),
\prod\limits\sb{l=J}\sp{N}(2p\sb l-1)
\right),
\end{gather}
where $p\sb{J}$ are exponentials from \eqref{f-is-such-no}.
Then for any $(\psi\sb{0},\pi\sb{0})\in\E$
the solution $\psi(t)$
to the Cauchy problem \eqref{kg-no}
converges to
$\mathcal{S}$:
\begin{gather}\label{cal-A-no}
\lim\sb{t\to\pm\infty}
\dist\sb{{\E}\sb F}((\psi(t),\dot\psi(t)),\mathcal{S})=0.
\end{gather}
\end{theorem}

\begin{theorem}[Global attraction for (\ref{kg-mf}),
Klein--Gordon equation with mean f\/ield interaction]
\label{main-theorem-mf}
Assume that the nonlinearity $F(z)$
is strictly nonlinear:
\begin{gather*}
F(z)=-\nabla U(z),
\qquad
{\rm where}
\quad
U(z)=\sum\sb{l=0}\sp{p}
u\sb l\abs{z}^{2l},
\quad
u\sb l\in\R,
\quad
u\sb{p}>0,
\quad
{\rm and}
\quad
p\ge 2.\!\!
\end{gather*}
Further, assume that the set
\begin{gather*}
Z\sb\rho
=\{\omega\in\R\backslash[-m,m]\sothat
\hat\rho(\xi)=0\ {\rm for\ all\ }\xi\in\R^n
\ {\rm such\ that\ }m^2+\xi^2=\omega^2\}
\end{gather*}
is finite
and that
\begin{gather*}
\sigma(\omega)
:=\frac{1}{(2\pi)^n}
\int\sb{\R^n}\frac{\abs{\hat\rho(\xi)}^2}{\xi^2+m^2-\omega^2}
\,d\sp n\xi
\end{gather*}
does not vanish
at the points $\omega\in Z\sb\rho$.
Then for any $(\psi\sb 0,\pi\sb 0)\in \E$
the solution $\psi(t)\in C(\R,\E)$
to the Cauchy problem \eqref{kg-mf}
converges to $\mathcal{S}$ in the space
$\E\sp{-\epsilon}\sb{F}$,
for any $\epsilon>0$:
\begin{gather}\label{cal-A-mf}
\lim\sb{t\to\pm\infty}
\dist\sb{\E\sp{-\epsilon}\sb{F}}(\Psi(t),\mathcal{S})=0.
\end{gather}
\end{theorem}

Above,
$\dist\sb{\E\sp{-\epsilon}\sb{F}}(\Psi,\mathcal{S})
:=\inf\limits\sb{\Phi\in\mathcal{S}}\norm{\Psi-\Phi}\sb{\E\sp{-\epsilon}\sb{F}}$,
with $\norm{\cdot}\sb{\E\sp{-\epsilon}\sb{F}}$
is def\/ined in  (\ref{def-e-metric}).

Theorem~\ref{main-theorem-no}
is proved in \cite{ukk-mpi};
Theorem~\ref{main-theorem-mf}
is proved in \cite{ukr-mpi}.
We present the sketch of the proof
of Theorem~\ref{main-theorem-no}
for one oscillator in Section~\ref{sect-proof}.

Let us mention several important points.

\begin{enumerate}\itemsep=0pt
\item
In the linear case,
the global attractor
contains the linear span of points of the solitary manifold,
$\langle\mathcal{S}\rangle$.
In \cite{ubk-arma},
we prove that for the model of one linear oscillator
attached to the Klein--Gordon f\/ield the global attractor
indeed coincides with $\langle\mathcal{S}\rangle$.
\item
The condition (\ref{delta-small})
allows to avoid ``trapped modes'',
which could also be characterized as multifrequency
solitary waves.
In Proposition~\ref{prop-delta-small} below,
we give an example of such solutions
in the situation
when the condition (\ref{delta-small}) is violated.

Similarly, the condition
of Theorem~\ref{main-theorem-mf}
that $\sigma(\omega)$
does not vanish for $\omega\in Z\sb\rho$
allows to avoid multifrequency solitary waves.
\item
We prove the attraction
of any f\/inite energy solution
to the solitary manifold $\mathcal{S}$:
\begin{gather}\label{attraction}
(\psi(t),\dot\psi(t))\longrightarrow\mathcal{S},
\qquad t\to\pm\infty,
\end{gather}
where the convergence holds in local seminorms.
In this sense,
$\mathcal{S}$ is a \emph{weak}
(convergence is local in space)
\emph{global}
(convergence holds for arbitrary initial data)
attractor.
\item
$\mathcal{S}$ can be at most a \emph{weak} attractor
because we need to
keep forgetting about the outgoing dispersive waves,
so that the dispersion plays the role of friction.
A \emph{strong} attractor would have
to consist of the direct sum of $\mathcal{S}$
and the space of outgoing waves.
\item
We interpret the local energy decay caused by dispersion
as a certain friction ef\/fect
in order to clarify the cause
of the convergence to the attractor in a Hamiltonian model.
This ``friction'' does not contradict the time reversibility:
if the system develops backwards in time,
one observes the same local energy decay
which leads to the convergence to the attractor
as $t\to -\infty$.
\item
Although we proved the attraction (\ref{attraction}) to $\mathcal{S}$,
we have not proved the attraction to a~particular solitary wave,
falling short of proving (\ref{asymptotics}).
Hypothetically,
if $\mathcal{S}/\mathbf{U}(1)$ contains
continuous components,
a solution can be drifting along $\mathcal{S}$,
keeping asymptotically close to it,
but never approaching a particular solitary wave.
This could be viewed as the adiabatic modulation
of solitary wave parameters.
Apparently, if $\mathcal{S}/\mathbf{U}(1)$ is discrete,
a~solution converges to a particular solitary wave.
\item
The requirement that the nonlinearity is polynomial
allows us to apply the Titchmarsh convolution theorem.
This step is vital in our approach.
We do not know whether the polynomiality requirement could be dropped.
\item
For the real initial data,
we obtain a real-valued solution $\psi(t)$.
Therefore, the conver\-gence~(\ref{cal-A-no}), (\ref{cal-A-mf})
of $\Psi(t)=(\psi(t),\dot\psi(t))$
to the set of pairs
$(\phi\sb\omega,-i\omega\phi\sb\omega)$
with $\omega\in\R\backslash \{0\}$
implies that $\psi(t)$
locally converges to zero.
\end{enumerate}

\subsection*{Sketch of the proof}

First, we introduce a concept of the omega-limit trajectory
$\beta(x,t)$ which
plays a crucial role in the proof.

\begin{definition}[Omega-limit trajectory]
\label{def-omega-limit}
The function $\beta(x,t)$
is an omega-limit trajectory
if there is a global solution
$\psi\in C(\R,\E)$
and a sequence of times $\{s\sb j\sothat j\in\N\}$
with
$\lim\limits\sb{j\to\infty}s\sb j=\infty$
so that
\[
\psi(x,t+s\sb j)\rightarrow\beta(x,t),
\]
where the convergence is in
$C\sb{b}([-T,T]\times\mathbb{B}^n\sb R)$
for any $T>0$ and $R>0$.
\end{definition}

We are going to prove that all omega-limit trajectories are solitary waves:
$\beta(x,t){=}\phi\sb\omega(x)e^{-i\omega t}$.
It suf\/f\/ices to prove that the time spectrum of
any omega-limit trajectory $\beta$
consists of at most one frequency.

To complete this program,
we study the time spectrum of solutions,
that is, their complex Fourier--Laplace transform in time.
First, we prove that
the spectral density of a solution
is absolutely continuous for $\abs{\omega}>m$
hence the corresponding component of the solution disperses completely.
It follows that the time-spectrum
of omega-limit trajectory $\beta$
is contained in a f\/inite interval $[-m,m]$.

Second, we notice that $\beta$ also satisf\/ies the original nonlinear equation.
Since the spectral support of $\beta$
is compact and the nonlinearity is polynomial,
we may apply the Titchmarsh convolution theorem.
This theorem allows to conclude that
the spectral support of the nonlinearity
would be strictly larger than the spectral support of the linear
terms in the equation (which would be a contradiction!)
except in the case when
the spectrum of the omega-limit trajectory
consists of a single frequency
$\omega\sb{+}\in[-m,m]$.

Since any omega-limit trajectory is a solitary
wave,
the attraction (\ref{attraction}) follows.

\subsection*{Open problems}

\begin{enumerate}\itemsep=0pt
\item
As we mentioned,
we prove the attraction to $\mathcal{S}$,
as stated in (\ref{attraction}),
but have not proved the attraction to a particular solitary wave
like (\ref{asymptotics}).
It would be interesting to f\/ind solutions with multiple
omega-limit points, that is, the situation when
the frequency parameter $\omega$ keeps changing adiabatically.
\item
Our argument does not apply to the Schr\"odinger equation.
The important feature of the Klein--Gordon equation is that the
continuous spectrum corresponds to $\abs{\omega}\ge m$,
hence the spectral density of the solution is absolutely continuous
for $\abs{\omega}\ge m$,
while the spectrum of the omega-limit trajectory
is within the compact set $[-m,m]$.
This is not so for the Schr\"odinger equation:
since the continuous spectrum corresponds to $\omega\ge 0$,
the resulting restriction on the spectrum of the omega-limit trajectory
is $\omega\le 0$.
As a result, we do not know whether the spectrum is compact;
the Titchmarsh convolution theorem does not
apply, and the proof breaks down.
It would be extremely interesting to investigate
whether the convergence to solitary waves is no longer true,
or instead certain modif\/ication of the Titchmarsh theorem
allows to reduce the spectrum to a point.
\item
Similarly,
the Titchmarsh theorem does not apply when the nonlinearity
is not polynomial, and it would be interesting to
investigate what could happen in such a case.
\end{enumerate}

\section[Proof
of attraction to solitary waves
for the Klein-Gordon field
with one nonlinear oscillator]{Proof
of attraction to solitary waves\\
for the Klein--Gordon f\/ield
with one nonlinear oscillator}
\label{sect-proof}

We will sketch the proof
of Theorem~\ref{main-theorem-no}
for the system (\ref{kg-1o})
which describes
one nonlinear oscillator
located at the origin.


\begin{proposition}[Compactness. Existence of omega-limit trajectories]
\label{coco} \
\begin{enumerate}\itemsep=0pt
\item
For any sequence $s\sb{j}\to+\infty$
there exists a subsequence $s\sb{j'}\to+\infty$
such that
\begin{gather}
\psi(x,s\sb{j'}+t)
\to
\beta(x,t),
\qquad
x\in\R,
\quad
t\in\R,
\label{ol}
\end{gather}
for some $\beta\in C(\R\times\R)$,
where the convergence
is in
$C\sb{b}([-T,T]\times[-R,R])$,
for any $T>0$ and $R>0$.
\item
\begin{gather}\label{beta-beta}
\sup\limits\sb{t\in\R}\norm{\beta(\cdot,t)}\sb{H\sp{1}}<\infty.
\end{gather}
\end{enumerate}
\end{proposition}

\begin{proof}
By Theorem~\ref{theorem-well-posedness}~({\it v}),
for any $\epsilon\in[0,1]$,
\begin{gather*}
\psi\in C\sp{(\epsilon)}(\R,H\sp{1-\epsilon}(\R)).
\end{gather*}
Taking $\epsilon=1/4$,
we see that
$\psi\in C\sp{(\alpha)}(\R\times\R)$,
for any $\alpha<1/4$.
Now the f\/irst statement of the proposition
follows by the Ascoli--Arzel\`a theorem.
The bound (\ref{beta-beta})
follows from
(\ref{ol}),
the bound (\ref{a-priori}),
and the Fatou lemma.
\end{proof}

We call {\it omega-limit trajectory}
any function $\beta(x,t)$
that can appear as a limit in (\ref{ol})
(cf.\ Def\/inition~\ref{def-omega-limit}).
We are going to prove that
every omega-limit trajectory $\beta$
belongs to the set of solitary waves;
that is,
\begin{gather}\label{eidd}
\beta(x,t)
=
\phi\sb{\omega\sb{+}}(x)e\sp{-i\omega\sb{+}t}
\qquad
{\rm for\ some} \quad \omega\sb{+}\in[-m,m].
\end{gather}

\begin{remark}\label{remark-end}
The fact that any omega-limit trajectory
turns out to be a solitary wave
implies the following statement:
\begin{gather}\label{cal-A-no-weaker}
{\rm if\ there\ is\ a\ sequence\  }
t\sb j\to\infty
{\rm \ \ so\ that\  }
(\psi(t\sb j),\dot\psi(t\sb j))\stackrel{\E\sb{F}}\longrightarrow
\Phi\in H\sp 1\times L\sp 2,
{\rm \ \ then\   }
\Phi\in\mathcal{S}.\!\!
\end{gather}
In turn,
(\ref{cal-A-no-weaker})
implies the convergence to the attractor
in the metric
(\ref{def-e-metric})
of $\E\sb{F}\sp{-\epsilon}$
for $\epsilon>0$:
\begin{gather*}
(\psi(t),\dot\psi(t))
\stackrel{\E\sb{F}\sp{-\epsilon}}\longrightarrow\mathcal{S},
\qquad
\epsilon>0,
\quad
t\to\pm\infty.
\end{gather*}
This is weaker than the convergence to the attractor
in the topology of $\E\sb{F}$
stated in Theorem~\ref{main-theorem-no}.
For the proof of the convergence to the attractor in the topology of $\E\sb{F}$,
see \cite{ubk-arma}.
\end{remark}

Let us split the solution
$\psi(x,t)$
into two components,
$\psi(x,t)=\chi(x,t)+\varphi(x,t)$,
which are def\/ined for all $t\in\R$
as solutions to the following Cauchy problems:
\begin{gather}
 \ddot\chi(x,t)=\chi''(x,t)-m^2\chi(x,t),
\qquad
(\chi,\dot\chi)\at{t=0}=(\psi\sb{0}(x),\pi\sb{0}(x)),
\nonumber
\\
\ddot\varphi(x,t)=\varphi''(x,t)-m^2\varphi(x,t)
+
\delta(x)f(t),
\qquad
(\varphi,\dot\varphi)\at{t=0}=(0,0),
\label{KG-cp-2-0}
\end{gather}
where
$(\psi\sb{0}(x),\pi\sb{0}(x))$ is the initial data from (\ref{kg-1o}),
and
\begin{gather}\label{def-f}
f(t):=F(\psi(0,t)),
\qquad
t\in\R.
\end{gather}

\begin{lemma}[Local energy decay of the dispersive component]
\label{lemma-decay-psi1}
There is the following decay for $\chi$:
\begin{gather*}
\lim\sb{t\to\infty}
\Norm{(\chi(\cdot,t),\dot\chi(\cdot,t))}\sb{{\E},R}=0,
\qquad\forall\, R>0.
\end{gather*}
\end{lemma}

For the proof, see \cite[Lemma 3.1]{ubk-arma}.
Lemma~\ref{lemma-decay-psi1}
means that the dispersive component
$\chi$ does not give any contribution
to the omega-limit trajectories
(see Def\/inition~\ref{def-omega-limit}).

\begin{figure}[htbp]
\centerline{\includegraphics{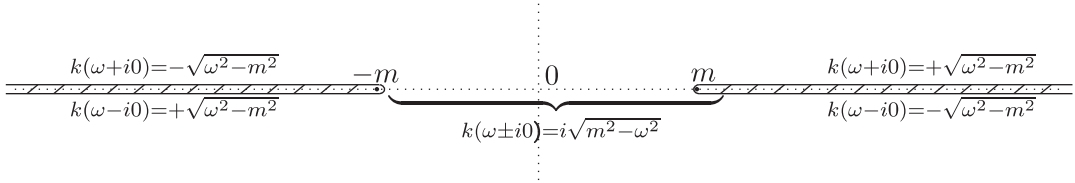}}

\caption{Domain $D$
and the values of
$k(\omega\pm i0)$, $\omega\in\R$.}
\label{fig-domain}
\end{figure}

Let $k(\omega)$ be the analytic function
with the domain
$D:=\C\backslash((-\infty,-m]\cup[m,+\infty))$
such that
\begin{gather}\label{def-k}
k(\omega)=\sqrt{\omega\sp 2-m^2},
\qquad\Im k(\omega)>0,
\qquad
\omega
\in D.
\end{gather}
See Fig.~\ref{fig-domain}.
Let us also denote the limit of $k(\omega)$
for $\omega+i0$, $\omega\in\R$, by
\begin{gather}\label{def-k-plus}
k\sb{+}(\omega):=k(\omega+i0),
\qquad
\omega\in\R.
\end{gather}

The function $\varphi(x,t)=\psi(x,t)-\chi(x,t)$
satisf\/ies the following Cauchy problem:
\begin{gather*}
\ddot\varphi(x,t)=\varphi''(x,t)-m^2\varphi(x,t)
+
\delta(x)f(t),
\qquad
(\varphi,\dot\varphi)\at{t=0}=(0,0),
\end{gather*}
with $f(t)$ def\/ined in (\ref{def-f}).
Note that
$\psi(0,\cdot)\in C\sb{b}(\R)$
by the Sobolev embedding,
since $(\psi,\dot\psi)\in C\sb{b}(\R,\E)$
by Theorem~\ref{theorem-well-posedness}~({\it iv}).
Hence, $f(t)\in C\sb{b}(\R)$.
On the other hand,
since $\chi(x,t)$
is a f\/inite energy solution to the free Klein--Gordon equation,
$(\chi,\dot\chi)\in C\sb{b}(\R,\E)$.
It follows that $\varphi=\psi-\chi$ is also of f\/inite energy norm:
\begin{gather}\label{psi-2-bounds}
(\varphi,\dot\varphi)\in C\sb{b}(\R,\E).
\end{gather}
We denote
\begin{gather*}
\varphi\sb{+}(x,t):=\theta(t)\varphi(x,t),
\qquad
f\sb{+}(t):=\theta(t)f(t)
=\theta(t)(\psi(0,t)).
\end{gather*}
The function $\varphi\sb{+}(x,t)$ satisf\/ies
the equation
\begin{gather}\label{KG-cp-2-pm}
\ddot\varphi\sb{+}(x,t)=\p\sb x^2\varphi\sb{+}(x,t)-m^2\varphi\sb{+}(x,t)
+
\delta(x)f\sb{+}(t),
\qquad
(\varphi\sb{+},\dot\varphi\sb{+})\at{t=0}
=(0,0),
\quad
t\in\R.
\end{gather}
We set $\mathcal{F}\sb{t\to\omega}[g(t)](\omega)
=\displaystyle\int\sb{\R} e^{i\omega t}g(t)\,dt$
for a function $g(t)$ from the Schwartz space $\mathscr{S}(\R)$.
The Fourier transform
\begin{gather*}
\hat\varphi\sb{+}(x,\omega)
=\mathcal{F}\sb{t\to\omega}[\varphi\sb{+}(x,t)]
=
\int\sb{0}\sp\infty
e^{i\omega t}
\varphi(x,t)\,dt,
\qquad
(x,\omega)\in\R^2,
\end{gather*}
is a continuous function of $x\in\R$
with
values in
tempered distributions of $\omega\in\R$,
which satisf\/ies the following equation (cf.~(\ref{KG-cp-2-pm})):
\begin{gather*}
-\omega^2\hat\varphi\sb{+}(x,\omega)
=
\p\sb x^2\hat\varphi\sb{+}(x,\omega)
-m^2\hat\varphi\sb{+}(x,\omega)
+\delta(x)\hat f\sb{+}(\omega),
\qquad
(x,\omega)\in\R^2.
\end{gather*}


\begin{proposition}[Spectral representation]
There is the following relation:
\begin{gather}\label{c-s}
\hat\varphi\sb{+}(x,\omega)
=
\hat\varphi\sb{+}(0,\omega)
e^{ik\sb{+}(\omega)\abs{x}},
\qquad
x\in\R.
\end{gather}
\end{proposition}

\begin{proof}
Let us analyze the complex Fourier transform of
$\varphi\sb{+}(x,t)$:
\begin{gather*}
\displaystyle
\tilde\varphi\sb{+}(x,\omega)
=\mathcal{F}\sb{t\to\omega}[\varphi\sb{+}(x,t)]
=
\int\sb{0}\sp\infty
e\sp{i\omega t}\varphi(x,t)\,dt,
\qquad
x\in\R,
\quad
\omega\in\C\sp{+},
\end{gather*}
where
$\C\sp{+}:=\{z\in\C:\;\Im z>0\}$.
Due to (\ref{psi-2-bounds}),
$\tilde\varphi\sb{+}(\cdot,\omega)$
are $H\sp{1}$-valued analytic functions of $\omega\in\C\sp{+}$.
Equation (\ref{KG-cp-2-pm})
implies that
$\tilde\varphi\sb{+}$ satisf\/ies
\begin{gather*}
-\omega\sp 2\tilde\varphi\sb{+}(x,\omega)
=
\p\sb x^2\tilde\varphi\sb{+}(x,\omega)-m^2\tilde\varphi\sb{+}(x,\omega)
+\delta(x)\tilde f\sb{+}(\omega),
\qquad\omega\in\C\sp{+}.
\end{gather*}
The fundamental solutions
$\displaystyle
G\sb\pm(x,\omega)=\frac{e\sp{\pm i k(\omega)\abs{x}}}{\pm 2i k(\omega)}$
satisfy
\[
G\sb\pm''(x,\omega)+(\omega^2-m^2)G\sb\pm(x,\omega)
=\delta(x),
\qquad\omega\in\C\sp{+}.
\]
Note that for each $\omega\in\C\sp{+}$ the function $G\sb{+}(\cdot,\omega)$
is in $H\sp{1}(\R)$
by def\/inition (\ref{def-k}),
while $G\sb{-}(\cdot,\omega)$ is not.
The solution $\tilde\varphi\sb{+}(x,\omega)$
can be written as a linear combination of these fundamental solutions.
We use the standard ``limiting absorption principle''
for the selection of the appropriate fundamental solution: Since
$\tilde\varphi\sb{+}(\cdot,\omega)\in H\sp{1}(\R)$
for $\omega\in\C\sp{+}$,
so is $G\sb{+}(\cdot,\omega)$,
while $G\sb{-}(\cdot,\omega)$ is not,
we have:
\begin{gather}\label{tilde-psi-tilde-f}
\tilde\varphi\sb{+}(x,\omega)
=-\tilde f\sb{+}(\omega)G\sb{+}(x,\omega)
=-\tilde f\sb{+}(\omega)
\frac{e\sp{i k(\omega)\abs{x}}}{2i k(\omega)},
\qquad\omega\in\C\sp{+}.
\end{gather}
The relation (\ref{tilde-psi-tilde-f}) yields
\begin{gather}\label{phi-1}
\tilde\varphi\sb{+}(x,\omega)
=-\tilde f\sb{+}(\omega)
\frac{e\sp{i k(\omega)\abs{x}}}{2i k(\omega)}
=e^{i k(\omega)\abs{x}}
\tilde\varphi\sb{+}(0,\omega),
\qquad
x\in\R,
\quad
\omega\in\C\sp{+}.
\end{gather}

Now we
extend the relation (\ref{phi-1})
to $\omega\in\R$.
Since
$\varphi\in C\sb{b}(\R,H\sp 1(\R))$
by (\ref{psi-2-bounds}),
we have
\begin{gather}\label{vpp-is-vpp}
\theta(t)\varphi(x,t)
=
\lim\sb{\varepsilon\to 0+}
\theta(t)\varphi(x,t)e\sp{-\varepsilon t},
\end{gather}
where the convergence holds in
the space of $H\sp{1}$-valued
tempered distributions,
$\mathscr{S}'(\R,H\sp{1}(\R))$.
The Fourier transform
$\hat\varphi\sb{+}(x,\omega)
:=\mathcal{F}\sb{t\to\omega}[\varphi\sb{+}(x,t)]
=\mathcal{F}\sb{t\to\omega}[\theta(t)\varphi(x,t)]$
is def\/ined
as a tempered $H\sp{1}$-valued distribution of $\omega\in\R$.
As follows from
(\ref{vpp-is-vpp})
and the continuity of the Fourier transform
$\mathcal{F}\sb{t\to\omega}$ in $\mathscr{S}'(\R)$,
$\hat\varphi\sb{+}(x,\omega)$
is the boundary value
of the analytic function $\tilde\varphi\sb{+}(x,\omega)$,
in the following sense:
\begin{gather}\label{bvp1}
\hat\varphi\sb{+}(x,\omega)
=\lim\limits\sb{\varepsilon\to 0+}
\tilde\varphi\sb{+}(x,\omega+i\varepsilon)
=\lim\limits\sb{\varepsilon\to 0+}
\mathcal{F}\sb{t\to\omega}[\theta(t)\varphi(x,t)e\sp{-\varepsilon t}],
\qquad\omega\in\R.
\end{gather}
Again, the convergence
is in the space
$\mathscr{S}'(\R,H\sp{1}(\R))$.

We use (\ref{bvp1})
to take the limit
$\Im\omega\to 0+$
in the expression (\ref{phi-1})
for $\tilde\varphi\sb{+}(x,\omega)$,
and keep in mind that
$\tilde\varphi\sb{+}(x,\omega)$
is a quasimeasure
(see Remark~\ref{remark-quasimeasure})
for each $x\in\R$,
while
the exponential factor in (\ref{phi-1})
is a multiplicator
in the space of quasimeasures.
The formula (\ref{c-s}) follows.
\end{proof}

\begin{remark}\label{remark-quasimeasure}
A tempered distribution $\mu(\omega)\in\mathscr{S}'(\R)$
is called a {\it quasimeasure} if
$\check\mu(t)
{=}\mathscr{F}\sp{-1}\sb{\omega\to t}[\mu(\omega)]$ $\in C\sb{b}(\R)$.
For more details
on quasimeasures and multiplicators
in the space of quasimeasures,
see \cite[Appendix B]{ubk-arma}.
\end{remark}

\begin{proposition}[Absolute continuity of the spectrum]
\label{prop-continuity-0}
The distribution
$\hat\varphi\sb{+}(0,\omega)$
is absolutely continuous for $\abs{\omega}>m$,
and moreover
\begin{gather}\label{phi-0-41}
\int\sb{\R\backslash[-m,m]}
\abs{\hat\varphi\sb{+}(0,\omega)}\sp 2
\frac{k\sb{+}(\omega)}{\omega}
\,d\omega<\infty,
\end{gather}
where
$k\sb{+}(\omega)/\omega>0$
for $\omega\in\R\backslash[-m,m]$
(see \eqref{def-k-plus}
and Fig.~{\rm \ref{fig-domain}}).
\end{proposition}


\begin{proof}
We use the Paley--Wiener arguments.
Namely, the Parseval identity and (\ref{psi-2-bounds})
imply that
\begin{gather}\label{PW}
\int\limits\sb\R
\norm{\tilde\varphi\sb{+}(\cdot,\omega+i\varepsilon)}\sb{L\sp 2}\sp 2\,d\omega
=2\pi\int\limits\sb 0\sp\infty e\sp{-2\varepsilon t}
\norm{\varphi\sb{+}(\cdot,t)}\sb{L\sp 2}\sp 2\,dt
\le\frac{\cnst}{\varepsilon},
\qquad \varepsilon>0.
\end{gather}
On the other hand,
we can calculate the term in the left-hand side of (\ref{PW}) exactly.
According to (\ref{phi-1}),
\[
\tilde\varphi\sb{+}(x,\omega+i\varepsilon)
=\tilde\varphi\sb{+}(0,\omega+i\varepsilon)e^{ik(\omega+i\varepsilon)\abs{x}},
\]
hence (\ref{PW}) results in
\begin{gather}\label{PW-1}
\varepsilon
\int\sb\R
\abs{\tilde\varphi\sb{+}(0,\omega+i\varepsilon)}^2
\norm{
e^{ik(\omega+i\varepsilon)\abs{x}}
}\sb{L\sp 2}\sp 2
\,d\omega
\le\cnst,
\qquad \varepsilon>0.
\end{gather}

Here is a crucial observation about the
norm of $e^{ik(\omega+i\varepsilon)\abs{x}}$.

{\samepage

\begin{lemma} \
\begin{enumerate}\itemsep=0pt
\item
For $\omega\in\R\backslash(-m,m)$,
\begin{gather}\label{wei}
\lim\sb{\varepsilon\to 0+}
\varepsilon \norm{e^{ik(\omega+i\varepsilon)\abs{x}}}\sb{L\sp 2}\sp 2
=
\frac{k\sb{+}(\omega)}{\omega}.
\end{gather}
\item
For any $\delta>0$
there exists $\varepsilon\sb\delta>0$ such that
for $\abs{\omega}>m+\delta$
and $\varepsilon\in(0,\varepsilon\sb\delta)$,
\begin{gather}\label{n-half}
\varepsilon \norm{e^{ik(\omega+i\varepsilon)\abs{x}}}\sb{L\sp 2}\sp 2
\ge\frac{k\sb{+}(\omega)}{2\omega}.
\end{gather}
\end{enumerate}
\end{lemma}}

\begin{remark}
The asymptotic behavior
of the $L\sp 2$-norm of
$e^{ik(\omega+i\varepsilon)}$
stated in the lemma
is easy to understand:
for $\omega\in\R\backslash[-m,m]$,
this norm is f\/inite
for $\varepsilon>0$
due to the small positive imaginary part
of $k(\omega+i\varepsilon)$,
but it becomes unboundedly large
when $\varepsilon\to 0+$.
Let us also mention
that the integral (\ref{wei})
is easy to evaluate
in the momentum space.
\end{remark}

Substituting
(\ref{n-half})
into (\ref{PW-1}), we get:
\begin{gather}\label{fin}
\int\sb{\abs{\omega}\ge m+\delta}
\abs{\tilde\varphi\sb{+}(0,\omega+i\varepsilon)}\sp 2
\frac{k\sb{+}(\omega)}{\omega}
\,d\omega
\le 2C,
\qquad
0<\varepsilon<\varepsilon\sb\delta,
\end{gather}
with the same $C$ as in (\ref{PW-1}).
We conclude that for each $\delta>0$ the set of functions
\[
g\sb{\delta,\varepsilon}(\omega)
=
\tilde\varphi\sb{+}(0,\omega+i\varepsilon)
\Abs{\frac{k\sb{+}(\omega)}{\omega}}^{1/2},
\qquad
\varepsilon\in(0,\varepsilon\sb\delta),
\]
def\/ined for $\omega\in\varOmega\sb\delta$,
is bounded in the Hilbert space
$L\sp 2(\R\backslash[-m-\delta,m+\delta])$,
and, by the Banach Theorem, is weakly compact.
The convergence of the distributions (\ref{bvp1})
implies the following weak convergence in the Hilbert space
$L\sp 2(\R\backslash[-m-\delta,m+\delta])$:
\begin{gather*}
g\sb{\delta,\varepsilon}
\rightharpoondown g\sb\delta,
\qquad \varepsilon\to 0+,
\end{gather*}
where the limit function
$g\sb\delta(\omega)$ coincides with the distribution
$
\hat\varphi\sb{+}(0,\omega)\Abs{\frac{k\sb{+}(\omega)}{\omega}}^{1/2}
$
restricted onto
$\R\backslash[-m-\delta,m+\delta]$.
It remains to note that,
by (\ref{fin}),
the norms
of all
functions $g\sb\delta$, $\delta>0$,
are bounded in $L\sp 2(\R\backslash[-m-\delta,m+\delta])$
by a constant independent on $\delta$,
hence (\ref{phi-0-41}) follows.
\end{proof}


By Lemma~\ref{lemma-decay-psi1},
the dispersive component $\chi(\cdot,t)$
converges
to zero in
${\E}\sb{F}$
as $t\to\infty$.
On the other hand,
by (\ref{ol}),
$\psi(x,t+s\sb{j'})$
converges to $\beta(x,t)$
as $j'\to\infty$,
uniformly on every compact set of the plane $\R^2$.
Hence,
$\varphi(x,t+s\sb{j'})
=
\psi(x,t+s\sb{j'})-\chi(x,t+s\sb{j'})$
also
converges to $\beta(x,t)$,
uniformly in every compact set of the plane $\R^2$:
\begin{gather}
\varphi(x,s\sb{j'}+t)
\to
\beta(x,t),
\qquad
x\in\R,
\quad
t\in\R.
\label{oll}
\end{gather}
Therefore, taking the limit in equation (\ref{KG-cp-2-0}),
we conclude that
the omega-limit trajectory $\beta(x,t)$ also satisf\/ies
the same equation:
\begin{gather*}
\ddot\beta(x,t)
=\beta''(x,t)-m^2\beta(x,t)
+\delta(x)F(\beta),
\qquad
x\in\R,
\quad
t\in\R.
\end{gather*}

Taking the Fourier transform of $\beta$ in time,
we see by (\ref{ol})
that $\hat\beta(x,\omega)$
is a continuous function of $x\in\R$,
with values in tempered distributions of $\omega\in\R$,
and that it satisf\/ies the corresponding stationary equation
\begin{gather}\label{KG-beta-1}
-\omega^2\hat\beta(x,\omega)
=\hat\beta''(x,\omega)-m^2\hat\beta(x,\omega)
+\delta(x)\hat g(\omega),
\qquad
(x,\omega)\in\R^2,
\end{gather}
valid in the sense of tempered distributions of $(x,\omega)\in\R^2$,
where
$\hat g(\omega)$
are the Fourier transforms of the function
\begin{gather*}
g(t):=F(\beta(0,t)).
\end{gather*}
For brevity, we denote
\begin{gather*}
\upbeta(t):=\beta(0,t).
\end{gather*}

\begin{lemma}[Boundedness of spectrum]
\label{lemma-m-m}
\[
\supp\hat\upbeta
\subset[-m,m].
\]
\end{lemma}

\begin{proof}
By (\ref{oll}), we have
\begin{gather}
\varphi\sb{+}(x,s\sb{j'}+t)
\to
\beta(x,t),
\qquad
x\in\R,
\quad
t\in\R,
\label{ol-plus}
\end{gather}
with the same convergence as in (\ref{ol}) and (\ref{oll}).
We have:
\[
\varphi\sb{+}(x,s\sb{j}+t)
=
\frac{1}{2\pi}\int\sb{\R}
e^{-i\omega t}
e^{-i \omega s\sb{j}}
\hat\varphi\sb{+}(x,\omega)\,d\omega,
\qquad
x\in\R,\quad t\in\R,
\]
where the integral is understood as the pairing
of a smooth function (oscillating exponent) with a compactly supported
distribution.
Hence, (\ref{ol-plus}) implies that
\begin{gather}\label{phi-to-beta}
e^{-i\omega s\sb{j'}}
\hat\varphi\sb{+}(x,\omega)
\to\hat\beta(x,\omega),
\qquad
x\in\R,
\quad s\sb{j'}\to\infty,
\end{gather}
in the sense of quasimeasures
(the convergence in the space
of quasimeasures is equivalent to the
Ascoli--Arzel\`a type convergence of
corresponding Fourier transforms;
see \cite[Appendix B]{ubk-arma}).
Since $\hat\varphi\sb{+}(0,\omega)$
is locally $L\sp 2$ for $\abs{\omega}>m$
by Proposition~\ref{prop-continuity-0},
the convergence (\ref{phi-to-beta})
at $x=0$
shows that
$\hat\upbeta(\omega):=\hat\beta(0,\omega)$
vanishes for $\abs{\omega}>m$.
This proves the lemma.
\end{proof}

We denote
\begin{gather}\label{def-kappa}
\kappa(\omega):=-i k\sb{+}(\omega),
\qquad
\omega\in\R,
\end{gather}
where $k\sb{+}(\omega)$ was introduced in (\ref{def-k-plus}).
We then have
$\Re\kappa(\omega)\ge 0$,
and also
\[
\kappa(\omega)=\sqrt{\omega^2-m^2}>0\qquad{\rm for}\quad -m<\omega<m,
\]
in accordance  with
(\ref{solitary-wave-profile}).

\begin{proposition}[Spectral representation for
$\beta$]
\label{prop-beta}
The distribution
$\hat\beta(x,\omega)$ admits the follo\-wing representation:
\begin{gather*}
\hat\beta(x,\omega)
=
\hat\upbeta(\omega)
e^{-\kappa(\omega)\abs{x}},
\qquad
x\in\R.
\end{gather*}
\end{proposition}

\begin{proof}
This follows by
taking the limit in the f\/irst line of (\ref{c-s}),
since
$\supp\hat\upbeta\subset[-m,m]$
by Lemma~\ref{lemma-m-m},
while
$k(\omega)=i\kappa(\omega)$
for $-m\le\omega\le m$
(cf.~(\ref{def-kappa})).
\end{proof}


\begin{proposition}[Reduction to point spectrum]
\label{prop-one-omega}
Either
$\supp\hat\upbeta=\{\omega\sb{+}\}$
for some $\omega\sb{+}\in[-m,m]$
or $\hat\upbeta=0$.
\end{proposition}

\begin{proof}
By Lemma~\ref{lemma-m-m},
we know that $\supp\hat\upbeta\subset[-m,m]$.
According to equation (\ref{KG-beta-1}),
the function
$\hat\beta$
satisf\/ies the following jump condition
at the point $x=0$:
\begin{gather*}
\hat\beta'(0+,\omega)-\hat\beta'(0-,\omega)
=
\hat g(\omega),
\qquad
\omega\in\R.
\end{gather*}
Since
$\supp\hat\beta'(0\pm,\cdot)\subset\supp\hat\upbeta$
by Proposition~\ref{prop-beta},
it follows that
\begin{gather}\label{supsup}
\supp\hat g(\cdot)
\subset\supp\hat\upbeta.
\end{gather}
On the other hand,
by (\ref{f-is-such-no}),
the Fourier transform
$\hat g(\omega)$
of $g(t):=F(\beta(0,t))$
is given by
\begin{gather}\label{conv}
\hat g
=
-\sum\sb{n=1}\sp{p}
2n\,u\sb{n}
\underbrace{
(\hat\upbeta\ast\hat{\overline\upbeta})
\ast\;\cdots\;\ast
(\hat\upbeta\ast\hat{\overline\upbeta})
}\sb{n-1}
\ast\hat\upbeta.
\end{gather}

Now we will use the Titchmarsh convolution theorem
\cite{titchmarsh} (see also \cite[p.~119]{MR1400006}
and \cite[Theorem 4.3.3]{MR1065136})
which could be stated as follows:
\begin{center}
{\it
For any compactly supported
distributions
$u$ and $v$,
$\sup\supp(u\ast v)=\sup\supp u+\sup\supp v$.
}
\end{center}
Applying the Titchmarsh convolution theorem
to the convolutions in (\ref{conv}),
we obtain the following equality:
\begin{eqnarray}
\sup\supp\hat g
\ge
\sup\supp\hat\upbeta
+(p-1)(\sup\supp\hat\upbeta-\inf\supp\hat\upbeta),
\label{supp-gg}
\end{eqnarray}
where we used the relation
$
\sup\supp\hat{\overline\upbeta}=-\inf\supp\hat\upbeta.
$
We wrote ``$\ge$''
because of possible cancellations
in the summation in the right-hand side of (\ref{conv}).
Note that the Titchmarsh theorem is applicable
to (\ref{conv})
since $\supp\hat\upbeta$ is compact
by Lemma~\ref{lemma-m-m}.

Comparing (\ref{supsup}) with (\ref{supp-gg}),
we conclude that
\begin{gather*}
(p-1)(\sup\supp\hat\upbeta-\inf\supp\hat\upbeta)=0.
\end{gather*}
Since $p\ge 2$
by (\ref{f-is-such-no})
(which means that the oscillator
at $x=0$ is nonlinear),
we conclude that
$\supp\hat\upbeta$ consists of at most a single point
$\omega\sb{+}\subset[-m,m]$.
\end{proof}

By Proposition~\ref{prop-one-omega},
$\supp\hat\upbeta\subset\{\omega\sb{+}\}$,
with $\omega\sb{+}\in[-m,m]$.
Therefore,
\begin{gather}\label{delom}
\hat\upbeta(\omega)=a\sb 1\delta(\omega-\omega\sb{+}),
\qquad {\rm with\ some}\ \ a\sb 1\in\C.
\end{gather}
Note that
the derivatives $\delta\sp{(k)}(\omega-\omega\sb{+})$, $k\ge 1$
do not enter the expression for
$\hat\upbeta(\omega)$
since $\upbeta(t)=\beta(0,t)$
is a bounded continuous function of $t$
due to the bound (\ref{beta-beta}).
Proposition~\ref{prop-beta}
and (\ref{delom})
imply that
the omega-limit trajectory $\beta(x,t)$
is a solitary wave:
\[
\beta(x,t)=\phi(x)e^{-i\omega\sb{+}t},
\]
where $\phi\in H\sp{1}(\R)$ by (\ref{beta-beta}).
This completes the proof of
(\ref{eidd}).

\begin{remark}
$\omega\sb{+}=\pm m$ could only correspond to the zero solution
by Remark~\ref{remark-zero}.
\end{remark}

\section{Multifrequency solitons}
\label{sect-examples}

\subsection{Linear degeneration}
\label{sect-example-li}
Let us consider equation (\ref{kg-no})
with $N=2$,
under condition (\ref{delta-small}).

\begin{proposition}\label{prop-nonlin}
If in \eqref{f-is-such-no}
one has $p\sb{J}=1$ for some $J$,
then the conclusion of Theorem~{\rm \ref{main-theorem-no}}
may no longer be correct.
\end{proposition}

\begin{proof}
We are going to construct the multifrequency solitary waves.
Consider the equation
\begin{gather*}
\ddot\psi
=\psi''-m^2\psi
+\delta(x)F\sb{1}(\psi)
+\delta(x-L)F\sb{2}(\psi),
\end{gather*}
where
\begin{gather*}
F\sb{1}(\psi)=\alpha\psi+\beta\abs{\psi}^2\psi,
\qquad
F\sb{2}(\psi)=\gamma\psi,
\qquad
\alpha,\ \beta,\ \gamma\in\R.
\end{gather*}
Note that the function $F\sb{2}$ is linear,
failing to satisfy (\ref{f-is-such-no}) (where one now has $p\sb 2=1$).
The function
\begin{gather*}
\psi(x,t)
=
\left\{\!\!
\begin{array}{l}
(A+B) e^{\kappa(\omega)x}\sin{\omega t},
\quad x\le 0,
\vspace{1mm}\\
\big(A e^{-\kappa(\omega)x}+B e^{\kappa(\omega)x}\big)
\sin{\omega t}
+C\sinh(\kappa(3\omega)x)\sin{3\omega t},
\quad x\in[0,L],
\vspace{1mm}\\
\big(A e^{-\kappa(\omega)}+B e^{\kappa(\omega)(2L-x)}\big)\sin{\omega t}
+\dfrac{C}{\sinh(\kappa(3\omega)L)}
e^{-\kappa(3\omega)(x-L)}\sin{3\omega t}
,
\quad
x\ge L,
\end{array}
\right.\!\!
\end{gather*}
where $\omega\in(0,m/3)$,
will be a solution if
the jump conditions are satisf\/ied
at $x=0$ and at $x=L$:
\begin{gather}\label{j-at-0}
-\psi'(0+,t)+\psi'(0-,t)=\alpha\psi(0,t)+\beta\psi^3(0,t),
\\
\label{j-at-l}
-\psi'(L+,t)+\psi'(L-,t)=\alpha\psi(L,t)+\beta\psi^3(L,t).
\end{gather}
Using the identity
\begin{gather}\label{trig}
\sin^3\theta=\frac 3 4\sin\theta-\frac 1 4\sin 3\theta,
\end{gather}
we see that
\begin{gather*}
\alpha(A+B)\sin{\omega t}
+\beta((A+B)\sin{\omega t})^3\\
\qquad{}
=\Big(\alpha(A+B)+\beta\frac{3(A+B)^3}{4}\Big)\sin{\omega t}
-\beta\frac{(A+B)^3}{4}\sin{3\omega t}.
\end{gather*}
Collecting the terms at $\sin{\omega t}$
and at $\sin{3\omega t}$,
we write the condition (\ref{j-at-0})
as the following system of equations:
\begin{gather}
2\kappa(\omega)A=\Big(\alpha(A+B)+\beta\frac{3(A+B)^3}{4}\Big),
\label{c01}
\\
-\kappa(3\omega)C=-\beta\frac{(A+B)^3}{4}.
\label{c03}
\end{gather}
Similarly, the condition (\ref{j-at-l}) is equivalent to the
following two equations:
\begin{gather}
2B\kappa(\omega)e^{\kappa(\omega)L}
=\gamma(A e^{-\kappa(\omega)L}+B e^{\kappa(\omega)L}),
\label{cl1}
\\
\frac{\kappa(3\omega)C}{\sinh(\kappa(3\omega)L)}
+
\kappa(3\omega)C\cosh(\kappa(3\omega)L)
=\gamma C\sinh(\kappa(3\omega)L).
\label{cl3}
\end{gather}
Equations (\ref{c01}), (\ref{c03}), (\ref{cl1}), and (\ref{cl3})
could be satisf\/ied for arbitrary $L>0$.
Namely, for any $\omega\in(0,m/3)$,
one uses (\ref{cl3}) to determine $\gamma$.
For any $\beta\ne 0$,
there is always a solution $A$, and $B$
to the nonlinear system (\ref{c01}), (\ref{cl1}).
Finally, $C$ is obtained from (\ref{c03}).
\end{proof}

\subsection{Wide gaps}
Let us consider equation (\ref{kg-no}) with $N=2$.
Assume that (\ref{f-is-such-no}) is satisf\/ied.

\begin{proposition}\label{prop-delta-small}
If the condition \eqref{delta-small}
is violated, then the conclusion of Theorem~{\rm \ref{main-theorem-no}}
may no longer be correct.
\end{proposition}

\begin{proof}
We will show that
if $L:=X\sb 2-X\sb 1$ is suf\/f\/iciently large,
then
one can take~$F\sb 1(\psi)$ and~$F\sb 2(\psi)$
satisfying
(\ref{f-is-such-no})
such that
the global attractor of the equation
contains the multifrequency solutions
which do not converge
to solitary waves of the form (\ref{solitary-waves}).
For our convenience, we assume that $X\sb{1}=0$, $X\sb{2}=L$.
We consider the model (\ref{kg-no}) with
\begin{gather*}
F\sb{1}(\psi)=F\sb{2}(\psi)
=F(\psi),
\qquad
{\rm where}
\quad
F(\psi)
=\alpha\psi+\beta\abs{\psi}^2\psi,
\qquad \alpha,\,\beta\in\R.
\end{gather*}
In terms of the condition (\ref{f-is-such-no}), $p\sb{1}=p\sb{2}=2$.
We take $L$ to be large enough:
\begin{gather}\label{m-pi}
L>\frac{\pi}{2\sp{3/2}m}.
\end{gather}
Consider the function
\begin{gather*}
\psi(x,t)=
A \big(e^{-\kappa(\omega)\abs{x}}+
e^{-\kappa(\omega)\abs{x-L}}\big)
\sin{\omega t}
+B\chi\sb{[0,L]}(x)\sin(k(3\omega) x)\,\sin{3\omega t},
\qquad
A,\,B\in\C.
\end{gather*}
Then
$\psi(x,t)$ solves (\ref{kg-no}) for $x$
away from the points $X\sb{J}$.
We require that
\begin{gather}\label{k-d}
k(3\omega)=\frac{\pi }{L},
\end{gather}
so that $\psi(x,t)$ is continuous in $x\in\R$
and symmetric with respect to $x=L/2$:
\[
\psi(x,t)
=
\psi\left(\frac{L}{2}-x,t\right),
\qquad x\in\R.
\]
We need
$\abs{\omega}<m$ to have $\kappa(\omega)>0$, and
$3\abs{\omega}>m$ to have $k(3\omega)\in\R$.
We take $\omega>0$, and thus
$m<3\omega<3m$.
By (\ref{k-d}), this means that we need
\[
m<\sqrt{\frac{\pi^2 }{L^2}+m^2}<3 m.
\]
The second inequality is satisf\/ied
by (\ref{m-pi}).

Due to the symmetry of $\psi(x,t)$
with respect to $x=L/2$,
the jump condition
both at $x=X\sb 1=0$ and at $x=X\sb 2=L$ takes
the following identical form:
\begin{gather}\label{j-at-0-l}
2A\kappa(\omega) \sin\omega t-B k(3\omega)\sin 3\omega t
=
F\big(A (1+e^{-\kappa(\omega)L})\sin{\omega t}\big).
\end{gather}
We use the following relation
which follows from (\ref{trig}):
\begin{gather*}
F\big(A (1+e^{-\kappa(\omega)L})\sin\omega t\big)
=
\Big(\alpha A(1+e^{-\kappa(\omega)L})
+
\frac{3}{4}
\beta \abs{A}^2 A(1+e^{-\kappa(\omega)L})^3
\Big)
\sin{\omega t}\nonumber
\\
\phantom{F\big(A (1+e^{-\kappa(\omega)L})\sin\omega t\big)=}{}
-\frac{1}{4}\beta \abs{A}^2 A(1+e^{-\kappa(\omega)L})^3
\sin{3\omega t}.
\end{gather*}
Collecting in (\ref{j-at-0-l})
the terms at $\sin\omega t$ and at $\sin 3\omega t$,
we obtain the following system:
\begin{gather}
2A\kappa(\omega)
=\alpha A(1+e^{-\kappa(\omega)L})
+\frac 3 4\beta \abs{A}^2 A(1+e^{-\kappa(\omega)L})^3,
\nonumber\\
B k(3\omega)=\frac 1 4\beta \abs{A}^2 A(1+e^{-\kappa(\omega)L})^3.\label{c-e}
\end{gather}
Assuming that $A\ne 0$,
we divide the f\/irst equation by $A$:
\begin{gather*}
2\kappa(\omega)
=\alpha (1+e^{-\kappa(\omega)L})
+\frac 3 4\beta \abs{A}^2(1+e^{-\kappa(\omega)L})^3.
\end{gather*}
The condition for the existence of a solution $A\ne 0$
is
\begin{gather}\label{2k-alpha}
\Big(\frac{2\kappa(\omega)}{1+e^{-\kappa(\omega)L}}-\alpha\Big)\beta>0.
\end{gather}
Once we found $A$,
the second equation
in (\ref{c-e})
can be used to express $B$ in terms of $A$.\end{proof}

\begin{remark}
Condition (\ref{2k-alpha}) shows that we can choose
$\beta<0$ taking large $\alpha>0$.
The corresponding potential
$U(\psi)=-\alpha\abs{\psi}^2/2-\beta\abs{\psi}^4/4$
satisf\/ies (\ref{f-is-such-no}).
\end{remark}


\pdfbookmark[1]{References}{ref}
\LastPageEnding

\end{document}